\documentclass[11pt,a4paper,twoside,reqno]{amsart}

\usepackage[hidelinks]{hyperref}

\usepackage[all]{xy}
\usepackage{graphicx}
\usepackage{epsfig, 
pstricks} 
\usepackage{float}
\usepackage{amsfonts, amscd}
\usepackage{amssymb}

\usepackage{mathrsfs}

\usepackage{fontenc}
\usepackage{amsmath}
\usepackage{latexsym}

\usepackage{tikz-cd}
 
\newcommand{\diam}{\operatorname{\mathrm{diam}}}
\newcommand{\sys}{\operatorname{\mathrm{sys}}}
\newcommand{\stsys}{\operatorname{\mathrm{stsys}}}
\newcommand{\PD}{\operatorname{\mathrm{PD}}}
\newcommand{\length}{\operatorname{length}}
\newcommand{\vol}{\operatorname{vol}}

\newcommand{\Tr}{\operatorname{\mathrm{Tr}}}

\newcommand{\st}{{\rm st}}

\newcommand*{\dd}%
  {\relax\ifnum\lastnodetype>0\mskip\medmuskip\fi\mathrm{d}}

\newcommand{\Z}{{\mathbb Z}}
\newcommand{\R}{{\mathbb R}}

\newcommand{\HH}{{\mathbb H}}

\newcommand{\T}{{\mathbb T}}

\newcommand{\Ric}{{\rm Ric}}
\newcommand{\LCD}{{\rm LCD}}

\newcommand{\sqrtR}{\sqrt{\rm R}{\rm ic}}

\newcommand{\cf}{{\it cf.}}
\newcommand{\ie}{{\it i.e.}}

\numberwithin{equation}{section}

\newtheorem{theorem}{Theorem}[section]
\newtheorem{proposition}[theorem]{Proposition}

\newtheorem{lemma}[theorem]{Lemma}

\theoremstyle{definition}
\newtheorem{definition}[theorem]{Definition}

\newtheorem{remark}[theorem]{Remark}

\long\def\forget#1\forgotten{} %

\title{Mixed sectional-Ricci curvature obstructions on tori}

\author[B.~Kloeckner]{Beno\^{\i}t Kloeckner}
\author[S.~Sabourau]{St\'ephane Sabourau}

\address{Universit\'e Paris-Est,
Laboratoire d'Analyse et Math\'ematiques Appliqu\'ees (UMR 8050), 
UPEC, UPEMLV, CNRS, F-94010, Cr\'eteil, France}

\email{benoit.kloeckner@u-pec.fr}
\email{stephane.sabourau@u-pec.fr}


\begin{document}

\begin{abstract}
We establish new obstruction results to the existence of Riemannian metrics on tori satisfying mixed bounds on both their sectional and Ricci curvatures.
More precisely, from Lohkamp's theorem, every torus of dimension at least three admits Riemannian metrics with negative Ricci curvature.
We show that the sectional curvature of these metrics cannot be bounded from above by an arbitrarily small positive constant.
In particular, if the Ricci curvature of a Riemannian torus is negative, bounded away from zero, then there exist some planar directions in this torus where the sectional curvature is positive, bounded away from zero.
All constants are explicit and depend only on the dimension of the torus.
\end{abstract}

\maketitle

\section{Introduction}

A classical line of research in Riemannian geometry is the relationship between the various curvatures of a Riemannian manifold and its underlying topological structure.
There are many results providing topological obstruction to the existence of metrics with sectional or Ricci curvature of a given sign.
Here are some of classical examples under one-sided curvature bounds among many others:
\begin{enumerate}
\item as a consequence of the Cartan-Hadamard theorem, a manifold whose universal covering is not diffeomorphic to $\mathbb{R}^n$ admits no complete metric of non-positive sectional curvature;
\item\label{enumi:Preissman} from Preissman's theorem, \cf~\cite[\S6.3.2, Theorem~24]{petersen}, the torus admits no Riemannian metric of negative sectional curvature;
\item as a consequence of Myers' theorem, \cf~\cite[\S6.4.1, Theorem~25]{petersen}, a closed manifold with infinite fundamental group admits no Riemannian metric of positive Ricci curvature;
\item finally, no closed manifold is almost nonnegatively Ricci curved unless its has an almost nilpotent fundamental group, \cf~\cite{CC96}.
\end{enumerate}

Actually, when constraining the Ricci curvature in Riemannian geometry, one is led to consider \emph{lower} bounds only since upper bounds do not carry much geometric or topological information.
More specifically, J.~Lohkamp showed that every closed manifold of dimension at least three admits a Riemannian metric of negative Ricci curvature, \cf~\cite{lo94}.
Moreover, these metrics are dense in the space of all Riemannian metrics, still in dimension at least three, for the $C^0$-topology, \cf~\cite{lo95}.
\medskip

Contrarily to previous works, we consider mixed curvature bounds involving \emph{upper} bounds on the Ricci curvature in this article.
More precisely, our main objective is to investigate the gap between the point~\ref{enumi:Preissman} above for Riemannian tori and Lohkamp's theorem. 
Although there are no metric with negative sectional curvature on any torus, tori of dimension at least three do admit metrics with negative Ricci curvature.
A natural question to ask is: given a Riemannian torus with negative Ricci curvature, what can be said about its sectional curvature?
In particular, can such a torus have almost nonpositive sectional curvature (that is, can its sectional curvature be bounded from above by an arbitrarily small positive constant)?
We negatively answer this question, showing a topological obstruction to such mixed curvature bounds.\footnote{Note that we use the convention $\Ric_M(u,u)=\Tr R(\cdot,u,\cdot,u)$ where $R$ is the Riemann curvature $(0,4)$-tensor, and that by $\Ric_M\le A$ we mean $\Ric_M(u,u)\le A\, g(u,u)$ for all $u\in TM$.}

\begin{theorem}\label{theo:main}
For all integer $n\ge 3$, there exist positive constants $\varepsilon_n$ and~$\Lambda_n$ such that for any $\varepsilon\in(0, \varepsilon_n)$, no Riemannian torus $M=(\mathbb{T}^n,g)$ can satisfy both
\[
K_M \cdot (\diam M)^2 \le \varepsilon \qquad\mbox{and}\qquad 
\Ric_M \cdot (\diam M)^2 \le - (n-1) \Lambda_n \, \varepsilon
\]
where $K_M$ denotes the sectional curvature and $\Ric_M$ is the Ricci curvature.

Moreover, one can take $\varepsilon_n =  2^{-6n^2-7n}$ and $\Lambda_n = 3000 \, n^5$.
\end{theorem}

To our knowledge, this is the first result of this kind involving an \emph{upper} bound on the Ricci curvature. 
Actually, such obstruction also holds for a more general class of manifolds, namely generalized torus, see~Definition~\ref{def:gen} and Theorem~\ref{theo:final}.

\medskip

While it could be tempting to use a compactness approach to prove this obstruction theorem (taking a sequence~$M_k=(\mathbb{T}^n, g_k)$ with $\sup K_{M_k}\to 0$ and Ricci curvature bounded from above by some negative constant, and passing to a limit to get a contradiction), such an argument seems doomed to fail under a Ricci curvature upper bound (and would not yield explicit bounds).

\medskip

The idea of the proof is to establish a double estimate on the relative volume growth of balls in the universal cover of a torus satisfying some upper bounds on the sectional and Ricci curvatures.
On the one hand, it follows from volume comparison estimates based on the ``root-Ricci curvature'' (a generalized G\"unther's inequality, \cf~\cite{KK15}) that  this growth is exponential for balls of not too large radius.
On the other hand, from a classical argument of Milnor \cite{Milnor68}, we know that the growth of the relative volume of balls is at most polynomial in their radius.
The strategy is then to combine the two growth estimates to show that too strong curvature bounds would lead to a contradiction. 
However, due to the finite range of the generalized G\"unther inequality, we need to make Milnor's argument non-asymptotic and to control the involved constants as independently of the metric as possible. 
A difficulty in this approach is to choose a good generating set for the fundamental group. One thing that could go wrong would be the presence of very short non-contractible loops.
This is a real possibility, dealt with by passing to a suitable finite cover of~$M$. 
The point is to kill the very short non-contractible loops by taking a finite cover with roughly the same diameter and a ``balanced shape". 
This leads us to the following result of independent interest: without any curvature assumption, up to a finite cover of controlled diameter, the stable norm of a torus can be non-asymptotically controlled by the displacement function (the precise definitions of these notions are given in Section~\ref{sec:cyclic-disp}), in terms of the dimension and the diameter only.
More precisely, we prove the following.

\begin{theorem} \label{theo:bd0}
Let $M$ be any Riemannian $n$-torus.
There exists an $n$-torus $N$ which is a finite Riemannian cover of $M$ with
\[
\diam(M) \le \diam(N) \le 6^n \, \diam(M)
\]
such that the displacement function~$\delta_N$ on~$N$ with respect to any basepoint satisfies
\begin{equation*}
\big\lvert \delta_N(\sigma) - \lVert\sigma\rVert_\st \big\rvert \le C_n \cdot \diam(N)
\end{equation*}
for every $\sigma \in H_1(N;\Z)$, where $C_n$ is an explicit constant depending only on~$n$. One can take $C_n=2^{4n^3+20 n^2}$.
\end{theorem}

Neither the value given for $C_n$, nor $6^n$ are optimal, and while we could state slightly better (but more complicated) bounds, it seems non-trivial to improve them dramatically.

\medskip

A version of this result, applying to $M$ directly without taking any cover, but with a metric-dependent constant~$C_M$ instead of $C_n\cdot \diam(N)$, has been established by D.~Burago~\cite{bur92}. 
Going over the argument, F.~Cerocchi and A.~Sambusetti \cite{CS16} showed that the metric-dependent constant~$C_M$ can be expressed in terms of the dimension~$n$, the diameter of~$M$ and the asymptotic volume of the universal cover of~$M$.
Actually, their result holds in the more general setting of length metric spaces with $\Z^n$-isometric actions.
They also presented an example in this setting showing that the difference between the displacement function and the stable norm cannot be bounded in terms of the dimension~$n$ and the (co)diameter only. 
Whether, in the case of a Riemannian torus, passing to a cover is necessary to obtain such a bound is a question left open\footnote{despite S.~Ivanov's suggestion on MathOverflow~\cite{IvMO} to adapt D.~Burago's original proof.}, but we do not need this for our purpose.
Actually, we do not use Theorem~\ref{theo:bd0} to prove Theorem~\ref{theo:main}, we rather deduce both results from a slightly different control of the displacement function involving word norms, see Theorem~\ref{theo:weak}.

\section{Cyclic covers and displacement} \label{sec:cyclic-disp}

Given a closed $n$-manifold~$M$, we shall construct particular cyclic covers of~$M$ and establish some cohomological properties related to their construction.

\medskip

There is a natural bijection between $H^1(M;\Z)$ and the set $[M,S^1]$ of homotopy classes of continuous maps from $M$ to the circle, provided by the fundamental relationship between cohomology and Eilenberg-MacLane spaces, \cf~\cite[Theorem~4.57]{hat}.
By definition, this bijection takes a continuous map $f:M \to S^1$ to the cohomology class $f^*(e) \in H^1(M;\Z)$, where $e$ is the fundamental cohomology class of~$S^1$ in $H^1(S^1;\Z)$.
Denote by $f_\lambda:M \to S^1$ a continuous map in the homotopy class induced by a cohomology class~$\lambda \in H^1(M;\Z)$ under this natural bijection. 

\medskip

We have the following relation.

\begin{lemma} \label{lem:primitive}
A cohomology class~$\lambda$ is primitive in~$H^1(M;\Z)$ if and only if the continuous map $f_\lambda:M \to S^1$ it induces is $\pi_1$-surjective.
\end{lemma}

\begin{proof}
Suppose that $\lambda$ is not primitive in~$H^1(M;\Z)$.
Then the class~$\lambda$ can be written as $\lambda = k \, \mu$, where $\mu \in H^1(M;\Z)$ and $k$ is a nonzero integer different from~$\pm 1$.
Since the multiplication by~$k$ in~$H^1(S^1;\Z)=\Z$ corresponds to a degree~$k$ covering~$\pi_k$ of~$S^1$ under the fundamental relationship between cohomology and Eilenberg-MacLane spaces, we derive the homotopy relation $f_\lambda \sim \pi_k \circ f_\mu$.
As the covering~$\pi_k$ is not $\pi_1$-surjective, the same holds for~$f_\lambda$.

\medskip

Conversely, suppose that the homomorphism $f_{\lambda *}:\pi_1(M) \to \pi_1(S^1)=\Z$ induced by~$f_\lambda$ is not $\pi_1$-surjective.
Then, unless this homomorphism is trivial, in which case $f_\lambda$ is homotopically trivial, its image ${\rm Im} f_{\lambda *}$ is an index~$k$ subgroup of~$\Z$ with $k > 1$.
From the covering theory, we derive that the map $f_\lambda:M \to S^1$ factors out through a degree~$k$ covering $\pi_k:S^1 \to S^1$.
That is, $f_\lambda = \pi_k \circ g$ for some continuous map $g:M \to S^1$.
In particular, 
\[
\lambda = f_\lambda^*(e) = g^*(\pi_k^*(e)) = k \, g^*(e).
\]
Thus, the class~$\lambda$ is not primitive in~$H^1(M;\Z)$.
\end{proof}

We are going to define a building block from which we will construct cyclic covers of~$M$.
We could directly define these cyclic covers, but their various properties are more convenient to establish by introducing this building block first.

\medskip

Let~$\lambda$ be a primitive cohomology class in~$H^1(M;\Z)$.The map $f_\lambda:M \to S^1$ induced by~$\lambda$ (and defined up to homotopy) lifts to a commutative diagram
\[\begin{tikzcd}
\bar{M}_\lambda \arrow[swap]{d}{\pi} \arrow{r}{\bar{f}_\lambda} 
  & S^1 \arrow{d}{} \\
M \arrow{r}{f_\lambda} & S^1
\end{tikzcd}\]
where the vertical maps are degree~$2$ cyclic covers.
The fixed-point free isometric involution of~$\bar{M}_\lambda$ corresponding to the free action of the nontrivial element of~$\Z_2$ on~$\bar{M}_\lambda $ is called the antipodal involution and will be denoted by~$\theta=\theta_\lambda$.

\medskip

Fix a point~$p$ in~$\bar{M}_\lambda$ and
let $q=\theta(p)$ be the image of~$p$ by the antipodal involution.
The Voronoi cell
\begin{equation} 
D' = \{ x \in \bar{M}_\lambda \mid d_{\bar{M}_\lambda }(x,p) \le d_{\bar{M}_\lambda }(x,q) \}.
\end{equation}
is a fundamental domain for the free action of~$\Z_2$ on~$\bar{M}_\lambda$. We will need to ensure some regularity of the boundary of the fundamental domain, and will thus slightly modifiy $D'$.

Consider a triangulation of $M$ including $\pi(p)$ as a vertex, with the associated PL structure (implicitly, we assume a PL triangulation, \ie, the link of each vertex is a PL sphere), and its lift to $\bar{M}_\lambda$.
We consider the PL functions $\bar{d}_p$ and $\bar{d}_q$ that match $d(p,\cdot)$ and $d(q,\cdot)$ on the $0$-skeleton, and we define
\begin{equation} \label{eq:Delta}
D= \{x \in \bar{M}_\lambda \mid \bar d_p(x) \le \bar d_q(x) \}
\end{equation}
which is a domain with PL boundary.
Choosing the triangulation fine enough, we can assume that $\bar{d}_p$ and $\bar{d}_q$ are uniformly close to $d(p,\cdot)$ and $d(q,\cdot)$ respectively. 
In particular, we can assume that $D$ is at Hausdorff distance at most $\frac{1}{10} \diam(M)$ from~$D'$.

Observe that since the triangulation on $\bar M_\lambda$ is a lift of a triangulation on $M$, we have $\bar d_p(x)=\bar d_q(\theta(x))$ and vice-versa. In particular, the double cover~$\bar{M}_\lambda$ is formed of the union of two isometric copies of~$D$, namely $D$ and~$\theta(D)$.

\begin{lemma} \label{lem:diam}
The adjusted Voronoi cell~$D$ centered at~$p$ satisfies
\begin{equation} \label{eq:diam}
\diam D \le 2.2 \diam M.
\end{equation}
\end{lemma}

\begin{proof}
Let $x \in D'$.
From the definition of the Voronoi cell~$D'$ and since $\pi^{-1}(\pi(p)) = \{p,q\}$, we have
\[
d_{\bar{M}_\lambda}(x,p) = d_{\bar{M}_\lambda}(x,\pi^{-1}(\pi(p))) = d_M(\pi(x),\pi(p)).
\]
Hence, $d(x,p) \le \diam M$. It follows that any $x\in D$ satisfies $d(x,p)\le 1.1\diam M$, and thus $\diam D\le 2.2\diam M$.
\end{proof}

Observe that the antipodal involution~$\theta$ leaves~$\partial D$ globally invariant.
More specifically, we have the following

\begin{lemma} \label{lem:antipodal}
The antipodal involution~$\theta$ takes every boundary component~$H_0$ of~$D$ to a boundary component of~$D$ different from~$H_0$.
\end{lemma}

\begin{proof}
We argue by contradiction.
Let $H_0$ be a boundary component of~$D$ such that~$\theta(H_0) = H_0$ and fix~$x \in H_0$. There are arcs $[px]$ and $[xq]$ lying in $D$ and $\theta(D)$ respectively. 
The arc $[px] \cup [xq]$ and its image by the antipodal map~$\theta$ form a $\Z_2$-invariant loop~$\gamma$ of~$\bar{M}_\lambda $.
As the lift~$\bar{f}_\lambda$ of the map~$f_\lambda$ is equivariant with respect to the antipodal involutions on~$\bar{M}_\lambda $ and~$S^1$, that is, $\Z_2$-equivariant, the image of this loop by~$\bar{f}_\lambda$ is noncontractible in~$S^1$.

\medskip

The loop~$\gamma$ intersects~$\partial D$ only twice, at the two points $x$ and~$\theta(x)$ of the connected component~$H_0$.
Thus, the loop~$\gamma$ is homotopic to the product of two loops lying on either side of~$\partial D$.
The images of these loops by~$\bar{f}_\lambda$ are contractible, so is the image of~$\gamma$.
Hence a contradiction with the previous claim.
\end{proof}

From Lemma~\ref{lem:antipodal}, the boundary~$\partial D$ of~$D$ decomposes into two disjoint isometric sets $H_+$ and~$H_-$, which are switched by the antipodal involution~$\theta$. 
That is,
\begin{equation} \label{eq:H}
\partial D = H_+ \cup H_-
\end{equation}
with $\theta$ taking $H_+$ to~$H_-$ and vice versa.
Now, given the Voronoi cell~$D$ defined in~\eqref{eq:Delta}, gluing back $H_+$ and~$H_-$ together gives rise to the original manifold~$M = D / \langle \theta \rangle$, while gluing two copies of~$D$ gives rise to the double cover $\bar{M}_\lambda =D \cup \theta(D)$.

%
%
%

\medskip
\
The $(n-1)$-cycle~$H_\lambda$ of~$M$ given by~$H_\pm$ is related to the cohomology class~$\lambda \in H^1(M;\Z)$ and the map $f_\lambda:M \to S^1$ through the Poincar\'e duality isomorphism as follows.

\begin{lemma} \label{lem:PD}
Up to the right choice of orientation of~$H_\lambda$, the Poincar\'e duality isomorphism 
\[
\PD:H^1(M;\Z) \to H_{n-1}(M;\Z)
\]
satisfies $\PD(\lambda) = [H_\lambda]$.

Moreover, we can assume that the map~$f_\lambda:M \to \R/\Z$ induced by~$\lambda$ satisfies $f_\lambda^{-1}(0) = H_\lambda$.
\end{lemma}

\begin{proof}
Though these relationships can be obtained from general abstract constructions, our argument will follow a hands-on approach. 

\medskip

The cohomology class 
\[
\mu \in H^1(M;\Z) = {\rm Hom}(H_1(M;\Z);\Z)
\]
 whose Poincar\'e dual is equal to~$[H_\lambda] \in H_{n-1}(M;\Z)$ is defined by taking the intersection with~$H_\lambda$.
More precisely, for every $1$-cycle~$c$ of~$M$, it is given by 
\[
\mu([c]) = [c] \cap [H_\lambda]
\]
where $\cap$ represents the (nondegenerate) intersection pairing in homology.

\medskip

Consider also a map $f:M \to S^1$ with $f^{-1}(0)=H_\lambda$ defined as follows.
Let $\bar{f}:D \to [0,\frac{1}{2}]$ be a continuous map with $\bar{f}^{-1}(0)=H_-$ and $\bar{f}^{-1}(\frac{1}{2})=H_+$.
Extend this map into a continuous map $\bar{f}:\bar{M}_\lambda \to S^1$, still denoted by~$\bar{f}$, equivariant under the antipodal involutions of~$\bar{M}_\lambda$ and $S^1=\R/\Z$.
That is, $\bar{f} \circ \theta = \bar{f} + \frac{1}{2}$.
By the equivariance property of~$\bar{f}$, this map descends to the desired map $f:M \to S^1$ with $f^{-1}(0) = H_\lambda$.

\medskip

Now, the $\theta$-invariant loop $[p,q] \cup \theta([p,q])$ of~$\bar{M}_\lambda$ formed of a minimizing segment~$[p,q]$ between $p$ and~$q$, and its image~$\theta([p,q])$ by~$\theta$ projects down to a simple closed curve of~$M$, denoted by~$c_0$.
The curve~$c_0$ has intersection~$\pm 1$ with~$H_\lambda$.
Furthermore, the map~$f$ takes the loop~$c_0$ to a generator of~$\pi_1(S^1)$.
Hence, $\mu(c_0) = \pm 1$ and $f^*(e)(c_0)=\pm 1$.

\medskip

Finally, let $H_+'$ be the connected $(n-1)$-cycle of~$\bar{M}_\lambda$ obtained from~$H_+$ through surgery by attaching thin tubes in~$D$ between the connected components of~$H_+$ so that $H_+'$ and its image~$H_-'=\theta(H_+')$ decompose~$\bar{M}_\lambda$ into two connected cells.
The projection of~$H_{\pm}'$ gives rise to a connected $(n-1)$-cycle~$H_\lambda'$ of~$M_\lambda$ homologous to~$H_\lambda$.
Consider the loops~$\gamma$ of~$M$ which do not intersect~$H_\lambda'$.
Since $H_\lambda'$ is connected, these loops generate the kernel of~$\mu$ and lift to~$\bar{M}_\lambda$.
Observe that their lifts are sent to contractible loops in~$S^1$ by the equivariant maps~$\bar{f}$ and~$\bar{f}_\lambda$.
Thus, the images of the loops~$\gamma$ under the maps~$f$ and~$f_\lambda$ are homotopically trivial.
Therefore the loops~$\gamma$ generate the kernels of~$f^*(e)$ and~$\lambda$.

\medskip

Since the cohomology classes~$\lambda$, $\mu$ and~$f^*(e)$ of
\[
H^1(M;\Z) = {\rm Hom}(H_1(M;\Z);\Z)
\]
have the same kernel, it follows that $\mu = k_1 \, \lambda$ and $f^*(e) = k_2 \, \lambda$ for some integers $k_1$ and~$k_2$ (recall that $\lambda$ is primitive).
As both classes $\mu$ and $f^*(e)$ take the values~$\pm 1$ at~$[c_0]$, we conclude that $\lambda$, $\mu$ and~$f^*(e)$ agree up to the sign.
In particular, the maps $f$ and~$f_\lambda$ are homotopic up to the sign.
\end{proof}

\begin{definition} \label{def:cyclic}
The construction of the double cover~$\bar{M}_\lambda$ by cutting~$M$ open along the $(n-1)$-cycle~$H_\lambda$ and gluing back two copies extends to any $\ell$-sheeted cyclic cover of~$M$, where $\ell$ is a positive integer which will be fixed later.
More precisely, we can arrange $\ell$ copies $D_1,\cdots,D_\ell$ of~$D$ in cyclic order to form a $\ell$-sheeted cyclic cover~$\hat{M}=\hat{M}_{\lambda,\ell}$ of~$M$ by identifying~$H_+^i$ to~$H_-^{i+1}$ (modulo~$\ell$), where $\partial D_i$ decomposes into $H_+^i \cup H_-^i$ as in~\eqref{eq:H}.
Denote by
\[
\pi:\hat{M} \to M.
\]
the corresponding degree~$d$ covering.
By construction, every copy~$D_i$ of~$D$ is a fundamental domain of~$M$ for the natural action of~$\Z_\ell$ on~$\hat{M}$.
\end{definition}

Let $\alpha$ be a primitive cohomology class in~$H^1(M;\Z)$.
The map \mbox{$f_\alpha:M \to S^1$} induced by~$\alpha$ lifts to a commutative diagram
\[\begin{tikzcd}
\hat{M}_{\lambda,\ell} \arrow[swap]{d}{\pi} \arrow{r}{\hat{f}_\alpha} 
  & S^1 \arrow{d}{\pi_\alpha} \\
M \arrow{r}{f_\alpha} & S^1
\end{tikzcd}\]
where $\hat{f}_\alpha:\hat{M} \to S^1$ is $\pi_1$-surjective and $\pi_\alpha:S^1 \to S^1$ is a covering of some positive degree~$m_\alpha$.
The $\pi_1$-surjective map~$\hat{f}_\alpha$ induces a primitive cohomology class $\hat{\alpha} \in H^1(\hat{M};\Z)$ under the fundamental relationship between cohomology and Eilenberg-MacLane spaces, \cf~Lemma~\ref{lem:primitive}.
That is, up to homotopy we can write $f_{\hat{\alpha}} = \hat{f}_\alpha$.

\begin{lemma} \label{lem:alpha}
The following relation holds
\[
\pi^*(\alpha) = m_\alpha \, \hat{\alpha}.
\]
Furthermore, $m_\lambda = \ell$.
\end{lemma}

\begin{proof}
By construction, we have
\[
f_\alpha \circ \pi = \pi_\alpha \circ f_{\hat{\alpha}}.
\]
At the degree one cohomology level, this relation yields
\[
\pi^*(f_\alpha^*(e)) = f_{\hat{\alpha}}^*(\pi_\alpha^*(e))
\]
where, as above, $e$ denotes the fundamental cohomology class of $S^1$.
Since $f_\alpha^*(e) = \alpha$ (and the same with~$\hat{\alpha}$) and the homomorphism $\pi_\alpha^*$ corresponds to the multiplication by~$m_\alpha$, this relation can be written as
\[
\pi^*(\alpha) = m_\alpha \, \hat{\alpha}.
\]
For $\alpha=\lambda$, the map $f_\lambda:M \to S^1$ with $f_\lambda^{-1}(0)=H_\lambda$ lifts to a $\pi_1$-surjective map $\hat{f}_\lambda:M \to S^1$ with $\hat{f}_\lambda^{-1}(\frac{i}{\ell})=H_+^{i-1}=H_-^i$.
It follows that the vertical map $\pi_\lambda:S^1 \to S^1$ is a degree~$\ell$ covering.
That is, $m_\lambda=\ell$.
\end{proof}

\section{Cohomology lengths and covers}

Using the same notations as in the previous section, we introduce cohomological lengths and study how the cohomological lengths of integral cohomology basis can grow by taking cyclic covers.

\medskip

Let us introduce the following definition.

\begin{definition} \label{def:gen}
A \emph{generalized $n$-torus} is a closed $n$-manifold~$M$ with fundamental group~$\Z^n$ such that the classifying map $\varphi:M \to \T^n$ to the $n$-torus has degree one.

For instance the connected sum of~$\T^n$ with any closed simply connected $n$-manifold is a generalized $n$-torus.

Note that every finite cover~$N$ of a generalized $n$-torus~$M$ is also a generalized $n$-torus.
Indeed, since every finite index subgroup of~$\Z^n$ is isomorphic to~$\Z^n$, we have $\pi_1(N) \simeq \Z^n$.
Moreover, the classifying map $\varphi:M \to \T^n$ lifts to a commutative diagram
\[\begin{tikzcd}
N \arrow{d} \arrow{r}
  & \T^n \arrow{d}{} \\
M \arrow{r}{\varphi} & \T^n
\end{tikzcd}\]
where the vertical maps are finite covers of degree $[\pi_1(M) : \pi_1(N)]$ and the horizontal map $N \to \T^n$ is the classifying map of~$N$, which is of degree one.
Here, we implicitly use the fact that every finite cover of~$\T^n$ is diffeomorphic to~$\T^n$.
\end{definition}

Consider a Riemannian generalized $n$-torus~$M$.
Given an integral cohomology basis $\alpha_1,\cdots,\alpha_n$ of~$H^1(M;\Z)$, an integer $\ell\ge 2$ and an index $k\in\{1,\dots,n\}$, consider the cohomology classes $\hat{\alpha}_1,\cdots,\hat{\alpha}_n \in H^1(\hat{M};\Z)$ induced by $\alpha_1,\cdots,\alpha_n$ on the $\ell$-sheeted cyclic cover $\hat{M}=\hat{M}_{\lambda,\ell}$ associated to~$\lambda := \alpha_k$, \cf~Section~\ref{sec:cyclic-disp}.
Note that $\hat{M}$ is a generalized $n$-torus.

\begin{lemma} \label{lem:basis}
The cohomology classes $\hat{\alpha}_1,\cdots,\hat{\alpha}_n$ form an integral cohomology basis of~$H^1(\hat{M};\Z)$.

Furthermore,
\[
m_{\alpha_i} = 
\begin{cases}
1 & \mbox{if } i \neq k \\
\ell & \mbox{if } i = k
\end{cases}
\]
\end{lemma}

\begin{proof}
Since the cohomology classes $\alpha_1,\cdots,\alpha_n$ form an integral cohomology basis of~$H^1(M;\Z)$ and the classifying map $\varphi:M \to \T^n$ induces an isomorphism between the first integral cohomology groups, the cohomology classes $\beta_1,\cdots,\beta_n$, where $\varphi^*(\beta_i)=\alpha_i$, also form an integral cohomology basis of~$H^1(\T^n;\Z)$ and their cup product is a generator of~$H^n(\T^n;\Z)$.
As the classifying map is of degree one, we have 
\[
\alpha_1 \cup \cdots \cup \alpha_n = \varphi^*(\beta_1 \cup \cdots \cup \beta_n) = \varphi_*(\omega_{\T^n}) = \omega_M
\]
for some orientations $\omega_{\T^n}$ and~$\omega_M$ of~$\T^n$ and~$M$.
Hence
\begin{equation} \label{eq:b1}
\pi^*(\alpha_1 \cup \cdots \cup \alpha_n) = \deg \pi \cdot \omega_{\hat{M}} = \ell \cdot \omega_{\hat{M}}.
\end{equation}
On the other hand, by Lemma~\ref{lem:alpha}, we have
\begin{equation} \label{eq:b2}
\pi^*(\alpha_1 \cup \cdots \cup \alpha_n) = \pi^*(\alpha_1) \cup \cdots \cup \pi^*(\alpha_n) = \left( \prod_{i=1}^n m_{\alpha_i} \right) \, \hat{\alpha}_1 \cup \cdots \cup \hat{\alpha}_n
\end{equation}
Combining~\eqref{eq:b1} and~\eqref{eq:b2}, we derive that there is an integer $p$ such that
\[
\hat{\alpha}_1 \cup \cdots \cup \hat{\alpha}_n = p \cdot  \omega_{\hat{M}} \qquad\mbox{and}\qquad p \cdot \prod_{i=1}^n m_{\alpha_i} = \ell.
\]
Since $m_{\alpha_k}=\ell$ from Lemma~\ref{lem:alpha}, we deduce that $p$ and all the integers~$m_{\alpha_i}$ with $i \neq k$ are equal to~$1$.
Hence, the cohomology classes $\hat{\alpha}_1,\cdots,\hat{\alpha}_n$ form an integral cohomology basis of~$H^1(\hat{M};\Z)$.
\end{proof}

\begin{definition} \label{def:ell}
Let $M$ be a Riemannian generalized $n$-torus.
For every nonzero cohomology class $\alpha \in H^1(M;\Z)$, define the \emph{cohomological length}~$K(\alpha)$ as the largest~$K \ge 0$ such that
\[
\length(c) \ge K \, \lvert \alpha(c) \rvert
\]
for every \emph{integral} one-cycle~$c$ in~$M$.
Since $\alpha$ is nonzero, the cohomological length~$K(\alpha)$ is well defined.

\medskip

Let $[c]$ denote the homology class of an integral one-cycle $c$.
By definition of the stable norm, \cf~Definition~\ref{def:displ}.\eqref{eq:stnorm}, and using the triangular inequality, we have $\length(c)\ge \lVert [c]\rVert_\st$. Letting $\lVert\cdot\rVert$ denote the cohomology norm dual to the stable norm in homology, we also have $\lvert\alpha(c)\rvert \le \lVert \alpha \rVert \cdot \lVert [c]\rVert_\st$. It follows that 
\[K(\alpha) = \inf_{c \neq 0} \frac{\length(c)}{\lvert \alpha(c)\rvert}  \ge \frac{1}{\lVert\alpha\rVert}\]
(and in particular $K(\alpha)>0$).

\medskip

Fix $\Delta>0$.
Consider the smallest nonnegative integer $\tau=\tau_M(\Delta)$ such that there exists an integral cohomology basis $\alpha_1,\cdots,\alpha_n \in H^1(M;\Z)$ with
\begin{equation} \label{eq:basis}
K(\alpha_1) \le \cdots \le K(\alpha_\tau) \le \Delta < K(\alpha_{\tau+1}) \le \cdots \le K(\alpha_n).
\end{equation}
\end{definition}

The following result is key in the proof of Theorem~\ref{theo:bd0}.

\begin{proposition} \label{prop:cover}
Let $M$ be a Riemannian generalized $n$-torus. 
Suppose \mbox{$\tau_M(\Delta) > 0$}.
Then there exists a cyclic Riemannian cover $\hat{M} \to M$ (where $\hat{M}$ is a generalized $n$-torus) with 
\[
\tau_{\hat{M}}(\Delta) \le \tau_M(\Delta) -1
\]
such that 
\[
\diam(\hat{M}) \le 0.5 \Delta + 5.5 \, \diam(M).
\]
\end{proposition}

\begin{proof}
Let $\alpha_1,\cdots,\alpha_n \in H^1(M;\Z)$ be an integral cohomology basis satisfying the inequality sequence~\eqref{eq:basis}.
Fix $\lambda=\alpha_\tau$, where $\tau=\tau_M(\Delta)$.
Let $\hat{M}=\hat{M}_{\lambda,\ell}$ be the $\ell$-cyclic cover induced by~$\lambda$, where $\ell$ will be fixed later, \cf~Definition~\ref{def:cyclic}.
From Lemma~\ref{lem:basis}, the integral cohomology basis $\alpha_1,\cdots,\alpha_n$ of~$H^1(M;\Z)$ gives rise to an integral cohomology basis $\hat{\alpha}_1,\cdots,\hat{\alpha}_n$ of~$H^1(\hat{M};\Z)$.

\medskip

Let $c$ be an integral one-cycle in~$\hat{M}$.
Since the map~$\pi$ is distance nonincreasing, we derive from the definition of~$K(\alpha_i)$ and the relation $\hat{\alpha}_i(c) = \alpha_i(\pi(c))$ given by the definition of~$\hat{\alpha}_i$ that
\[
\length(c) \ge \length(\pi(c)) \ge K(\alpha_i) \big\lvert\alpha_i(\pi(c))\big\rvert \ge K(\alpha_i) \big\lvert\hat{\alpha}_i(c)\big\rvert.
\]
Hence
\begin{equation} \label{eq:Ki}
K(\hat{\alpha}_i) \ge K(\alpha_i).
\end{equation}

By construction, the cyclic cover~$\hat{M}$ is formed of $\ell$ copies $D_1,\cdots,D_\ell$ of~$D$ glued together in cyclic order with $H_-^{i+1} = H_+^i$ (modulo~$\ell$) where $\partial D_i = H_-^i \cup H_+^i$, \cf~Section~\ref{sec:cyclic-disp}.
Cutting open~$\hat{M}$ along $H_-^1=H_+^\ell$ gives rise to an $n$-manifold $\hat{M}_0 = D_1 \cup \cdots \cup D_\ell$ with $\partial \hat{M}_0=H_-^1 \cup H_+^\ell$.
Clearly, 
\[
d_{\hat{M}_0}(H_-^1,H_+^\ell) \ge \ell \, d_D(H_+,H_-).
\]

Now, let $\ell$ be the minimal integer such that $d_{\hat{M}_0}(H_-^1,H_+^\ell) > \Delta$.
As the cohomology class~$\hat{\lambda}$ is Poincar\'e dual to the homology class of $H_-^1=H_+^1$ in~$\hat{M}$, \cf~Lemma~\ref{lem:PD}, every integral one-cycle~$c$ in~$\hat{M}$ with $\hat{\lambda}(c) \neq 0$ runs across~$\hat{M}_0$ between $H_-^1$ and~$H_+^\ell$ at least $|\hat{\lambda}(c)|$ times.
Therefore,
\[
\length(c) > \Delta \cdot |\hat{\lambda}(c)|.
\]
Hence, 
\begin{equation} \label{eq:Ktau}
K(\hat{\lambda}) = K(\hat{\alpha}_\tau) > \Delta.
\end{equation}

From the inequalities~\eqref{eq:Ki} and~\eqref{eq:Ktau}, we immediately deduce that
\[
\tau_{\hat{M}}(\Delta) \le \tau_M(\Delta) -1.
\]

Now, we need to bound from above the diameter of~$\hat{M}$.
First, observe that $\ell \ge 2$, otherwise $\hat{M}=M$ which would contradict the previous inequality.

\medskip

Let $\gamma$ be the shortest loop in~$\hat{M}$ with $\hat{\lambda}(\gamma) = 1$.
Since the loop~$\gamma$ goes through every cell~$D_i$ of~$\hat{M}$, we immediately deduce that
\begin{equation}\label{eq:g-}
\diam(\hat{M}) \le \frac{1}{2} \length(\gamma) + 2 \, \diam(D)
\end{equation}
by connecting every point of~$\hat{M}$ to a point of~$\gamma$ at distance at most~$\diam(D)$.

\medskip

Let us now bound from above the length of $\gamma$.
Let $\hat{M}'=\hat{M}'_{\lambda,\ell-1}$ be the $(\ell-1)$-cyclic cover induced by~$\lambda$ (recall that $\ell \ge 2$) and denote by~$\hat{\lambda}'$ the cohomology class of~$H^1(\hat{M}';\Z)$ induced by~$\lambda$, \cf~Section~\ref{sec:cyclic-disp}.
By minimality of~$\ell$, there exists an arc~$\gamma'$ in $\hat{M}'_0=D_1 \cup \cdots \cup D_{\ell-1}$ connecting the $(n-1)$-cycles $H_-^1$ and~$H_+^{\ell-1}$ with $\length(\gamma') \le \Delta$.

\medskip

The endpoints of~$\gamma'$ lie in
\[
\partial D_\ell = H_-^\ell \cup H_+^\ell = H_+^{\ell-1} \cup H_-^1
\]
and can be joined by an arc~$\gamma''$ of~$D_\ell$ with $\length(\gamma'') \le \diam(D)$.

\medskip

By construction, the loop~$\gamma' \cup \gamma''$ of~$\hat{M}$ has intersection~$1$ with~$H_-^1=H_+^\ell$, that is, $\hat{\lambda}(\gamma' \cup \gamma'') = 1$.
Therefore, by definition of~$\gamma$, we obtain
\begin{align}
\length(\gamma) & \le \length(\gamma') +\length(\gamma'') \notag \\
 & \le \Delta + \diam(D).
\label{eq:g+}
\end{align}

\medskip

Combining the inequalities \eqref{eq:g-}, \eqref{eq:g+} and~\eqref{eq:diam}, we immediately derive the diameter bound of Proposition~\ref{prop:cover}.
\end{proof}

\section{Displacement and homology norms}

Relying on the results of the previous sections, we establish two quantitative versions of the Bounded Distance Theorem, \cf~\cite{bur92}, up to finite index integral lattices.

\medskip

We first need to recall a few definitions.

\begin{definition} \label{def:displ}
Let $M$ be a closed Riemannian manifold.
The first integral homology group $H_1(M,\Z)$  isometrically acts on the universal Riemannian cover~$\tilde M$ of~$M$.
Fix $\tilde{x}_0$ a basepoint in~$\tilde{M}$.
The \emph{displacement function}~$\delta_M$ with respect to~$\tilde{x}_0$ is defined as
\begin{equation} \label{eq:displ}
\delta_M(\sigma) = d_{\tilde{M}}(\sigma . \tilde{x}_0,\tilde{x}_0)
\end{equation}
for every $\sigma \in H_1(M;\Z)$.
In other words, the displacement~$\delta_M(\sigma)$ of~$\sigma$ represents the length of the shortest loop~$\gamma$ of~$M$ based at the projection~$x_0$ of~$\tilde{x}_0$ representing~$\sigma$ (\ie, $\sigma=[\gamma]$).

\medskip

The \emph{stable norm} on~$H_1(M;\Z)$ is given by
\begin{equation} \label{eq:stnorm}
\lVert \sigma \lVert_\st = \lim_{k \to +\infty} \frac{\delta_M(k \, \sigma)}{k}.
\end{equation}

It can be extended by homogeneity and uniform continuity to~$H_1(M;\R)$.
The stable norm on~$H_1(M;\R)$ agrees with the norm induced by the Riemannian length in homology with real coefficients.
That is,
\[
\lVert \sigma \rVert_\st = \inf \left\{ \sum_i \lvert a_i\rvert \, \length(\gamma_i) \right\}
\]
where the infimum is taken over all the real linear combinations of Lipschitz $1$-cycles $\sum_i a_i \, \gamma_i$ representing~$\sigma$ in~$H_1(M;\R)$.

\medskip

The (homology) \emph{systole} and the \emph{stable systole} of~$M$ are defined as
\[
\sys(M) = \inf_{\tilde{x} \in \tilde{M}}\; \inf_{\substack{\vphantom{\tilde{M}} \sigma \in H_1(M;\Z) \\ \sigma\neq 0}} d_{\tilde{M}}(\sigma . \tilde{x},\tilde{x})
\]
and 
\[
\stsys(M) = \inf_{\substack{\sigma \in H_1(M;\Z) \\ \sigma\neq 0}} \lVert\sigma\rVert_\st.
\]
\end{definition}

We can now establish the weaker version of the quantitative Bounded Distance Theorem from which we will deduce the stronger version given by Theorem~\ref{theo:bd}.

\begin{theorem} \label{theo:weak}
Let $M$ be a Riemannian generalized $n$-torus.
There exist a generalized $n$-torus $N$ which is a finite Riemannian cover of $M$ with
\[
\diam(M) \le \diam(N) \le 6^n \, \diam(M)
\]
and a word norm~$\lVert \cdot \rVert $ on~$H_1(N;\Z)$ with respect to an integral homology basis such that the displacement function~$\delta_N$ on~$N$ with respect to any basepoint, \cf~Definition~\ref{def:displ}.\eqref{eq:displ}, satisfies
\begin{equation} \label{eq:dis}
A_n \cdot \diam(N) \cdot \lVert\sigma\rVert  \le \delta_N(\sigma) \le B_n \cdot \diam(N) \cdot  \lVert\sigma\rVert 
\end{equation}
for every $\sigma \in H_1(N;\Z)$, where one can take 
\[A_n=\frac{1}{n \, 6^n} \qquad\mbox{and}\qquad B_n=2^{3n^2-2n}.\]
\end{theorem}

\begin{remark}
Note that by homogeneity of these bounds, the displacement function can be replaced by the stable norm.
\end{remark}

\begin{proof}
Fix $\Delta=\diam(M)$.
Applying Proposition~\ref{prop:cover} at most $n$ times gives rise to a finite Riemannian cover $N \to M$ with $\tau_N(\Delta)=0$.
This cover~$N$ is a generalized $n$-torus, \cf~Definition~\ref{def:gen}.
Since every time we apply Proposition~\ref{prop:cover} the diameter of the cover increases by a factor at most~$6$, we immediately derive that 
\begin{equation} \label{eq:6nd}
\diam(N) \le 6^n \, \diam(M).
\end{equation}

Now, as $\tau_N(\Delta)$ is equal to zero, there exists an integral cohomology basis $\alpha_1,\cdots,\alpha_n \in H^1(N;\Z)$ with $K(\alpha_i) > \diam(M)$ for every~$i$.
Let $e_1,\cdots,e_n$ be the integral homology basis of~$H_1(N;\Z)$ dual to the cohomology basis~$\alpha_1,\cdots,\alpha_n$ of~$H^1(N;\Z)$, that is, $\alpha_i(e_j) = \delta_{i,j}$.
Every integral homology class~$\sigma \in H_1(N;\Z)$ decomposes as
\[
\sigma = \sum_{i=1}^n \alpha_i(\sigma) \, e_i.
\]
Thus, if $\lVert \cdot \rVert_1$ and $\lVert \cdot \rVert_\infty$ represent the $\ell^1$-norm and the sup-norm on~$H_1(N;\Z)$ with respect to the basis~$(e_i)$, then
\begin{equation} \label{eq:sigmainf}
\lVert\sigma\rVert_1 = \sum_{i=1}^n |\alpha_i(\sigma)| \quad \mbox{ and } \quad \lVert\sigma\rVert_\infty = \max_{1 \le i \le n} |\alpha_i(\sigma)|.
\end{equation}
Observe also that 
\begin{equation} \label{eq:Lnorms}
\lVert\sigma\rVert_\infty \le \lVert\sigma\rVert_1 \le n \, \lVert\sigma\rVert_\infty.
\end{equation}

\medskip

Consider the displacement function~$\delta_N$ on~$N$ associated to any basepoint~$\tilde{x}_0$ in~$\tilde{N}$, \cf~Definition~\ref{def:displ}.\eqref{eq:displ}.
Denote by~$x_0$ the projection of~$\tilde{x}_0$ to~$N$.
Since $K(\alpha_i) \ge \diam(M)$, every loop~$\gamma$ of~$N$ based at~$x_0$ with~$\sigma = [\gamma]$ satisfies
\[
\diam(M) \cdot |\alpha_i(\sigma)| \le \length(\gamma)
\]
for every~$i$ (simply recall that $\alpha_i(\sigma) = \alpha_i(\gamma)$).
Combined with~\eqref{eq:6nd}, \eqref{eq:sigmainf} and~\eqref{eq:Lnorms}, this inequality leads to
\[
\frac{1}{n \, 6^n} \cdot \diam(N) \cdot \lVert\sigma\rVert_1 \le \delta_N(\sigma).
\]

\medskip

Let us show that a reverse inequality holds up to a multiplicative constant.
From~\cite[Lemma 2]{Milnor68}, the fundamental group~$\pi_1(N,x_0)$ of~$N$, and so~$H_1(N;\Z)$, is generated by loops of length at most~$2 \, \diam(N)$.
Thus, there exist $n$~loops $\gamma_1,\cdots,\gamma_n$ of~$N$ based at~$x_0$ generating an integral homology basis of~$H_1(N;\Z)$ with
\begin{equation} \label{eq:gj+}
\length(\gamma_j) \le 2 \, \diam(N) \le 2 \cdot 6^n \cdot \diam(M).
\end{equation}
Since $K(\alpha_i) \ge \diam(M)$, we also have the following lower bound
\begin{equation} \label{eq:gj-}
\diam(M) \cdot |\alpha_i(\gamma_j)| \le \length(\gamma_j). 
\end{equation}
Combining the previous inequalities~\eqref{eq:gj+} and~\eqref{eq:gj-}, we derive
\begin{equation} \label{eq:gj}
|\alpha_i(\gamma_j) | \le 2 \cdot 6^n.
\end{equation}

\medskip

The integral homology basis $e_1,\cdots,e_n$ and $[\gamma_1],\cdots,[\gamma_n]$ of~$H_1(N;\Z)$ are related by
\[
[\gamma_i] = \sum_{j=1}^n \alpha_j(\gamma_i) \, e_j.
\]
The transformation matrix~$A$ between these two integral homology basis is the invertible integral matrix of size~$n$ with coefficients~$\alpha_j(\gamma_i)$.
Since the integral matrix~$A$ is invertible, the determinant of~$A$ equals~$\pm 1$.
From~\eqref{eq:gj}, the coefficients~$\alpha_j(\gamma_i)$ of~$\Delta$ are bounded by~$2 \cdot 6^n$ in absolute value.
By the Hadamard inequality,
\[
|\det(A')| \le \prod_{j=1}^{n-1} \lVert C_j\rVert_2
\]
where the $C_j$ are the columns of any square matrix~$A'$ of size~$n-1$, it follows that the cofactors of~$A$ are bounded by 
\[
c_n=(2 \cdot \sqrt{n-1} \cdot 6^n)^{n-1}
\]
in absolute value.
By the inverse formula,
\[
A^{-1} = \frac{1}{\det(A)} \, {}^t {\rm com}(A)
\]
where ${\rm com}(A)$ is the comatrix of~$A$, the same bound holds in absolute value on the coefficients of the inverse transformation matrix $A^{-1}=B=(b_{i,j})$, where $b_{i,j} \in \Z$.
That is,
\[
|b_{i,j}| \le c_n.
\]

\medskip

Now, from the relation 
\[
e_i = \sum_{j=1}^n b_{i,j} \, [\gamma_j]
\]
we can construct a representative of the integral homology class~$e_i$ by concatenating the loops~$\gamma_j$ based at the same point~$x_0$.
From the first length estimate of~\eqref{eq:gj+} and the definition of the displacement function, \cf~Definition~\ref{def:displ}.\eqref{eq:displ}, we conclude that the displacement of~$e_i$ satisfies
\begin{align*}
\delta_N(e_i) & \le \sum_{j=1}^n b_{i,j} \, \length(\gamma_j) \\
 & \le 2 \, c_n \cdot \diam(N)
\end{align*}

From the triangle inequality, we finally deduce the following upper bound
\begin{align*}
\delta_N(\sigma) & \le \sum_{i=1}^n |\alpha_i(\sigma)| \, \delta_N(e_i) \\
 & \le 2 \, c_n \cdot \diam(N) \cdot \lVert \sigma\rVert_1
\end{align*}
As the $\ell^1$-norm~$\lVert  \cdot \rVert_1$ agrees with the word norm on~$H_1(N;\Z)$ with respect to the integral homology basis $e_1,\cdots,e_n$, this finishes the proof of the double inequality~\eqref{eq:dis}, with a constant
\[\bar B_n = 2^n(n-1)^{\frac{n-1}{2}}6^{n(n-1)}.\]
We can finally simplify this to the larger $B_n=2^{3n^2-2n}$ (take $\log_2 \bar B_n$ and use $\log_2 x\le 0.6 x$, $\log_2 6\le 2.6$).
\end{proof}

Relying on this result, we can now establish the following stronger version of Theorem~\ref{theo:weak}.

\begin{theorem} \label{theo:bd}
Let $M$ be a Riemannian generalized $n$-torus.
There exists a generalized $n$-torus $N$ which is a finite Riemannian cover of $M$ with
\[
\diam(M) \le \diam(N) \le 6^n \, \diam(M)
\]
such that the displacement function~$\delta_N$ on~$N$ with respect to any basepoint satisfies
\begin{equation*}
\left| \delta_N(\sigma) - \lVert \sigma\rVert_\st \right| \le C_n \cdot \diam(N)
\end{equation*}
for every $\sigma \in H_1(N;\Z)$, where one can take $C_n=2^{4n^3+20 n^2}$.
\end{theorem}

\begin{proof}
The inequality~\eqref{eq:dis} implies that the stable systole of~$N$ can be bounded from below in terms of the diameter of~$N$.
More precisely,
\[
A_n \cdot \diam(N) \le \stsys(N) \le \sys(N) \le 2 \, \diam(N).
\]

Indeed, from the first inequality in~\eqref{eq:dis} and the homogeneity of~$\lVert  \cdot \rVert_1$, we have
\[
A_n \cdot \diam(N) \cdot k \, \lVert \sigma\rVert_1 \le \delta_N(k \, \sigma)
\]
for every $\sigma \in H_1(N;\Z)$.
From the definition of the stable norm~\eqref{eq:stnorm} and the relation~$\lVert \sigma\rVert_1 \ge 1$ holding for every~$\sigma \neq 0$, we deduce that
\[
A_n \cdot \diam(N) \le \lVert  \sigma \rVert_\st
\]
for every nonzero $\sigma \in H_1(N;\Z)$.
Hence the lower bound on the stable systole.

\medskip

From~\cite{bur92}, we know that the displacement function~$\delta_N$ of a Riemannian torus~$N$ is at bounded distance from the stable norm.
The dependence of this distance in terms of global geometric invariant of~$N$ has been made explicit in~\cite{CS16}.
Namely, it only depends on the dimension~$n$, an upper bound on the diameter and a lower bound on the stable systole.
In our case, since the diameter and the stable systole only differ by a multiplicative constant depending only on~$n$, this bounded distance theorem can be expressed as follows
\[
| \delta_N(\sigma) - \lVert \sigma\rVert_\st | \le \bar C_{n} \cdot \diam(N)
\]
for every $\sigma \in H_1(N;\Z)$, where $\bar C_{n}$ is an explicit constant depending only on~$n$.
The formulas given in~\cite{CS16} show that we can take
\[
\bar C_{n}=2^{n^2+6n+10} n^2(n!)^{n+2}\Big(\frac{2^n}{A_n^n}+1\Big)^{n+4}
\]
Using the value $A_n=\frac{1}{n 6^n}$, we can thus take 
\[
\bar C_{n}=2^{n^2+6n+10} n^2(n!)^{n+2}(2^n\cdot 6^{n^2}\cdot n^n+1)^{n+4}
\]
Taking the $\log_2$ of this expression and using $n\ge 3$, $\log_2n \le 0.53 n$ and \mbox{$n!\le n^n$}, we get that one can take the larger but simpler value $C_n=2^{4n^3+20 n^2}$.
\end{proof}

\begin{remark}
Note that even the non-simplified constant is pretty large:
$\bar C_3\ge 2\cdot 10^{81}$.
\end{remark}

\section{Volume growth of balls}

In this section, we establish a double estimate on the relative volume of balls in the universal cover of a generalized torus satisfying some upper bounds on the sectional and Ricci curvatures.
As a consequence, we derive the main theorem of this article.

\medskip

Standard comparison arguments between quasi-isometric distances and norms lead to the following polynomial relative volume comparison estimate, which does not require any curvature bound.

\begin{proposition} \label{prop:rupper}
Let $N$ be a Riemannian generalized $n$-torus.
Fix a point~$\tilde{x}_0$ in the universal Riemannian cover~$\tilde{N}$ of~$N$.
Suppose that the displacement function~$\delta_N$ of~$N$ satisfies the double inequality~\eqref{eq:dis}.
Then for every reals~$r,R$ with $R \ge r > a_n \, \diam(N)$,
\[
\frac{\vol B(\tilde{x}_0,R)}{\vol B(\tilde{x}_0,r)} \le D_n  \left( \frac{R+\diam(N)}{r-\diam(N)} \right)^n
\]
where $a_n$ and $D_n$ are explicit constants depending only on~$n$. More precisely, one can take
\[a_n = 8^{n^2} \qquad D_n = 2^{3n^3+2n^2} \]
\end{proposition}

\begin{proof}
Consider the Voronoi cells~$D_\sigma$ formed of the points~$x$ of~$\tilde{N}$ closer to~$\sigma . \tilde{x}_0$ than any other point of the orbit of~$\tilde{x}_0$ under the action of~$H_1(N;\Z)$ on~$\tilde{N}$.
That is,
\[
D_\sigma = \{ x \in \tilde{N} \mid d_{\tilde{N}}(x,\sigma . \tilde{x}_0) \le d_{\tilde{N}}(x,\sigma' \! . \tilde{x}_0) \mbox{ for every } \sigma' \neq \sigma \in H_1(N;\Z) \}
\]
Note that the ball of radius~$\Delta = \diam(N)$ centered at~$\sigma . \tilde{x}_0$ covers the Voronoi cell~$D_\sigma$.

\medskip

Let $r > (3n \, B_n +1) \, \diam(N)$, where $B_n$ is the explicit constant given by Theorem~\ref{theo:weak} (this enables us to choose $a_n=8^{n^2}$).
On the one hand, the ball~$B(\tilde{x}_0,r)$ contains the union of the Voronoi cells centered at the points of the orbits of~$\tilde{x}_0$ lying in~\mbox{$B(\tilde{x}_0,r-\Delta)$}.
On the other hand, the ball~$B(\tilde{x}_0,r)$ is contained in the union of the Voronoi cells centered at the points of the orbits of~$\tilde{x}_0$ lying in~$B(\tilde{x}_0,r+\Delta)$.
Therefore,
\[
\# \{ \sigma \mid \delta_N(\sigma) \le r - \Delta \} \cdot \vol(N) \le \vol B(\tilde{x}_0,r) \le \# \{ \sigma \mid \delta_N(\sigma) \le r + \Delta \} \cdot \vol(N).
\]
From Theorem~\ref{theo:weak}, this yields
\[
\# \Big\{ \sigma \,\Big|\, B_n \cdot \lVert \sigma\rVert_1 \le \frac{r - \Delta}{\Delta} \Big\} \le \frac{\vol B(\tilde{x}_0,r)}{\vol(N)} \le \# \Big\{ \sigma \,\Big|\, A_n \cdot \lVert \sigma\rVert_1 \le \frac{r + \Delta}{\Delta} \Big\}.
\]

Let $B_1(t)$ and $B_\infty(t)$ be the balls of radius~$t$ for the $\ell^1$ and~$\ell^\infty$ norms on~$H_1(N;\Z)$.
Straighforward estimates, whose proof we will omit, show that their cardinalities satisfy
\[
\Big(1.6\frac{t}{n}\Big)^n \le | B_\infty(\tfrac{t}{n}) |  \le | B_1(t) | \le | B_\infty(t) | \le (2.2 \, t)^n
\]
for every $t \ge 3n$.

Our assumption on~$r$ allows us to apply these estimates, which leads to the double inequalities
\[
A_n' \left( \frac{r-\Delta}{\Delta} \right)^n \le \frac{\vol B(\tilde{x}_0,r)}{\vol(N)} \le B_n'  \left( \frac{r+\Delta}{\Delta} \right)^n
\]
where $A_n'=\frac{1.6}{(n\, B_n)^n}$ and $B_n'=(\frac{2.2}{A_n})^n$.
Hence
\[
\frac{\vol B(\tilde{x}_0,R)}{\vol B(\tilde{x}_0,r)} \le \bar D_n  \left( \frac{R+\Delta}{r-\Delta} \right)^n.
\]
where 
\[\bar D_n=\frac{B'_n}{A'_n} = 1.375^n \cdot n^{2n} \cdot 2^{3n^3-2^{n^2}} \cdot 6^{n^2} \le 2^{3n^3+2n^2} =:D_n,\]
again using a rough logarithm estimate.
\end{proof}

\begin{remark}
By taking a finite cover of~$N$ as in Theorem~\ref{theo:weak}, we can easily assume that the assumption on the displacement function is satisfied.
\end{remark}

Let us recall some volume comparison estimates related to the root-Ricci function, \cf~\cite{KK15}, leading to an exponential growth of the relative volume of balls at small scale.

\begin{definition}
Let $M$ be a complete Riemannian $n$-manifold with sectional curvature~$K_M$.
Let $\rho$ be a nonnegative real such that $K_M \le \rho$.
For every point~$p \in M$ and every unit tangent vector~$u \in U_p M$, we define the \emph{root-Ricci function} as
\[
\sqrtR(\rho,u) = \Tr(\sqrt{\rho \, \langle \cdot , \cdot \rangle - R(\cdot,u,\cdot,u)})
\]
where $R(\cdot,\cdot,\cdot,\cdot)$ is the Riemann curvature tensor expressed as a tetralinear form and the square root is the positive square root of a positive semidefinite operator.

\medskip

A complete Riemannian $n$-manifold~$M$ is of $\sqrtR$ \emph{class}~$(\rho,\kappa)$ with $\rho \ge 0$ and $\kappa \le \rho$ if the two following conditions are satisfied
\begin{align*}
K_M & \le \rho \\
\frac{\sqrtR(\rho,u)}{n-1} & \ge \sqrt{\rho-\kappa}
\end{align*}
for every $u \in UM$.
\end{definition}

The following upper bounds on the sectional and Ricci curvatures easily imply the $\sqrtR$ class conditions.
\begin{lemma}[\cite{KK15}\footnote{Note that the formula stated in the published version of \cite{KK15}, bottom of page 125, is different from this: a mistake slipped through in the translation from the formula giving $\beta(\kappa,\alpha,\rho)$ (which is correct) to the formula giving $\beta(\kappa,\rho,\rho)$, where the constant $n(n-1)$ should in fact be $(n-1)(n-2)$.}] \label{lem:LCD}
Let $M$ be a complete Riemannian $n$-manifold.
If $K_M \le \rho$ and $\Ric_M \le - (n-1)\lambda$ with $\rho, \lambda \ge 0$ then the manifold~$M$ is of $\sqrtR$ \emph{class}~$(\rho,\kappa)$ with 
\[
\kappa = - \frac{1}{n-1} \, \lambda + \frac{n-2}{n-1} \, \rho.
\]
\end{lemma}

The $\sqrtR$ class conditions lead to volume comparison estimates through the notion of candle function.

\begin{definition}
Let $M$ be a complete Riemannian manifold.
Denote by~$\gamma_u$ the geodesic of~$M$ starting at~$p$ with initial velocity~$u \in U_p M$.
Define the \emph{candle function} $s(r,u)$ of~$M$ as the Jacobian of the exponential map
\[
u \mapsto \gamma_u(r) = \exp_p(ru).
\]
In other words, the candle function~$s$ is given by the equation 
\[
\dd v = s(r,u) \dd r \dd u
\]
where $\dd v$ is the Riemannian volume on~$M$, $\dd u$ is the Riemannian volume on the round unit sphere~$U_p M$ and $\dd r$ is the Lebesgue measure on~$\R$.

\medskip

The candle function~$s_{n,\kappa}$ of a complete Riemannian $n$-manifold of constant curvature~$\kappa$ does not depend on the unit vector~$u$.
It can be expressed as
\[
s_{n,\kappa}(r) = 
\begin{cases}
\left( \frac{\sin(\sqrt{\kappa} \, r)}{\sqrt{\kappa}} \right)^{n-1} & \mbox{if } \kappa > 0 \mbox{ and } r \le \frac{\pi}{\sqrt{\kappa}} \\[1ex]
r^{n-1} & \mbox{if } \kappa = 0 \\
\left( \frac{\sinh(\sqrt{-\kappa} \, r)}{\sqrt{-\kappa}} \right)^{n-1} & \mbox{if } \kappa < 0
\end{cases}
\]
We can set $s_{n,\kappa}(r)=0$ when $\kappa > 0$ and $r \ge \frac{\pi}{\sqrt{\kappa}}$.

\medskip

Let $\ell \ge 0$ (and assume that $\ell \le \frac{\pi}{\sqrt{\kappa}}$ when $\kappa >0$).
A complete Riemannian $n$-manifold~$M$ satisfies the \emph{logarithmic candle derivative} condition $\LCD(\kappa,\ell)$~if 
\[
\left( \log s(r,u) \right)' \ge  \left( \log s_{n,\kappa}(r) \right)' 
\]
for every $u \in UM$ and $0 \le r \le \ell$.
Here, the derivatives are taken with respect to~$r$.
\end{definition}

\begin{theorem}[\cite{KK15}] \label{theo:KK}
Let $M$ be a complete Riemannian $n$-manifold of $\sqrtR$ \emph{class}~$(\rho,\kappa)$ with $\rho \ge 0$ and $\kappa \le \rho$.
Then $M$ is $\LCD \left(\kappa,\frac{\pi}{2\sqrt{\rho}} \right)$.
\end{theorem}

Arguing as in the proof of the Bishop-Cheeger-Gromov inequality (see Lemma~1.6 in~\cite[\S9.1.3]{petersen}) we obtain the following relative volume comparison result.

\begin{proposition} \label{prop:rcomp}
Let $M$ be a complete simply connected Riemannian $n$-manifold of $\sqrtR$ \emph{class}~$(\rho,\kappa)$ with $\rho \ge 0$ and $\kappa \le \rho$.
Then
\[
\frac{\vol B(p,R)}{\vol B(p,r)} \ge \frac{\vol B_\kappa(R)}{\vol B_\kappa(r)}
\]
for every $p \in M$ and $0 \le r \le R \le \frac{\pi}{2\sqrt{\rho}}$, where $B(p,r)$ is the ball of radius~$r$ centered at~$p$ in~$M$ and $B_\kappa(r)$ is the ball of radius~$r$ in the model $n$-space~$\HH_\kappa^n$ of constant sectional curvature~$\kappa$.
\end{proposition}

\begin{proof}
Fix $\ell = \frac{\pi}{2\sqrt{\rho}}$.
As $K_M \le \rho$, the Rauch theorem, \cf~\cite[\S6.6.1]{petersen}, implies that the exponential map defines a diffeomorphism from $B(0,r) \subset T_pM$ onto~$B(p,r) \subset M$ for every $r \le \ell$.
Under this change of variable, we can write
\[
\vol B(p,r) = \int_0^r \int_{U_p M} s(t,u) \dd u \dd t.
\]
Introduce the function~$f$ defined by
\[
f(r) = \log \left( \frac{\vol B(p,r)}{\vol B_\kappa(r)} \right).
\]
Its derivative can be written as
\[
f'(r) = \frac{a(r)}{\int_0^r a(t) \dd t} - \frac{b(r)}{\int_0^r b(t) \dd t}
\]
where 
\[
a(t) = \int_{U_p M} s(t,u) \dd u \quad \mbox{ and } \quad b(t) = \int_{S^{n-1}} s_{n,\kappa}(t) \dd u.
\]
Therefore, the sign of~$f'(r)$ is given by the sign of
\[
\int_0^r \big( a(r) \, b(t) - a(t) \, b(r) \big) \dd t.
\]
Thus, $f'(r) \ge 0$ if the function~$\frac{a(r)}{b(r)}$ is nondecreasing.

Now, observe that
\[
\frac{a(r)}{b(r)} = \frac{\int_{U_p M} s(r,u) \dd u}{\int_{S^{n-1}} s_{n,\kappa}(r) \dd u} = \frac{1}{\vol S^{n-1}} \, \int _{U_p M} \frac{s(r,u)}{s_{n,\kappa}(r)} \dd u.
\]
where $\vol S^{n-1}$ is the Euclidean volume of the unit sphere in~$\R^n$.
By assumption, the function $r \mapsto \log \left( \frac{s(r,u)}{s_{n,\kappa}(r)} \right)$  is nondecreasing 
and so is the function~$r \mapsto \frac{s(r,u)}{s_{n,\kappa}(r)}$.
Integrating this expression over~$U_p M$ with respect to~$u$, we deduce that the same holds for $\frac{a(r)}{b(r)}$.

Therefore, the function~$f$ is nondecreasing on~$[0,\ell]$, which implies the desired inequality.
\end{proof}

Fix $\kappa< 0$.
Recall that the ball~$B_\kappa(R)$ of radius~$R$ in the model $n$-space~$\HH_\kappa^n$ of constant curvature~$\kappa$ has volume
\begin{equation} \label{eq:formvol}
\vol B_\kappa(R) = \vol S^{n-1} \cdot \int_0^R \left( \frac{\sinh(\sqrt{|\kappa|} \, t)}{\sqrt{|\kappa|}} \right)^{n-1} \, dt
\end{equation}

We deduce that the growth of the relative volume of balls in $\HH_\kappa^n$ is exponential. The precise formulation of the result below is far from optimal, but our goal is to have it simple and explicit.
\begin{proposition} \label{prop:rlower}
We have
\[
\frac{\vol B_\kappa(R)}{\vol B_\kappa(r)} \ge \frac{1}{2^{2n-1}} \cdot \left( \frac{R}{r} \right)^n \cdot e^{(n-1) \sqrt{|\kappa|} \, (\frac{R}{4}-r)}. 
\]
\end{proposition}

\begin{proof}
Since $\sinh(u) \le u \, e^u$ for every $u \ge 0$, we derive from~\eqref{eq:formvol} that
\[
\vol B_\kappa(r) \le \vol S^{n-1} \cdot \int_0^r t^{n-1} \, e^{(n-1) \sqrt{|\kappa|} \, t} \, dt.
\]
As the integrand is nondecreasing, it follows
\begin{equation}
\vol B_\kappa(r) \le \vol S^{n-1} \cdot r^n \cdot e^{(n-1) \sqrt{|\kappa|} \, r}.
\label{eq:lowerexp}
\end{equation}
Since $\sinh(u) \ge \frac{u}{2} \, e^\frac{u}{2}$ for every $u \ge 0$, we derive from~\eqref{eq:formvol} that
\[
\vol B_\kappa(R) \ge \frac{\vol S^{n-1}}{2^{n-1}}  \cdot \int_0^R t^{n-1} \, e^{(n-1) \sqrt{|\kappa|} \, \frac{t}{2}} \, dt.
\]
As the integrand is nonnegative and nondecreasing, the integral can be bounded from below by
\[
\int_{\frac{R}{2}}^R t^{n-1} \, e^{(n-1) \sqrt{|\kappa|} \, \frac{t}{2}} \, dt \ge \frac{R}{2} \cdot \left( \frac{R}{2} \right)^{n-1}  e^{(n-1) \sqrt{|\kappa|} \, \frac{R}{4}}.
\]
Thus, we obtain
\begin{equation}
\vol B_\kappa(R) \ge \frac{\vol S^{n-1}}{2^{2n-1}} \cdot R^n \cdot e^{(n-1) \sqrt{|\kappa|} \, \frac{R}{4}}.
\label{eq:upperexp}
\end{equation}

The result follows from \eqref{eq:lowerexp} and \eqref{eq:upperexp}.
\end{proof}

We can now prove our main result, which we restate here for convenience.

\begin{theorem} \label{theo:final}
There exist explicit positive constants $\varepsilon_n$ and~$\Lambda_n$, such that for all $\varepsilon\in(0, \varepsilon_n)$, no Riemannian generalized $n$-torus $M$ can satisfy both conditions
\[K_M \cdot (\diam M)^2 \le \varepsilon \qquad\mbox{and}\qquad 
\Ric_M \cdot (\diam M)^2 \le -(n-1)\Lambda_n \, \varepsilon\]
More precisely, one can take
\[\varepsilon_n =  2^{-6n^2-7n} \qquad \mbox{and} \qquad \Lambda_n = 3000 \, n^5.\]
\end{theorem}

\begin{proof}
Let $\varepsilon>0$ and assume 
\[K_M \cdot (\diam M)^2 \le \varepsilon \qquad\mbox{and}\qquad \Ric_M \cdot(\diam M)^2 \le -(n-1)\lambda \, \varepsilon\] for some positive constant~$\lambda$. We need to prove that if $\varepsilon$ is small enough, the factor~$\lambda$ cannot be too large.

Consider a generalized $n$-torus~$N$ covering~$M$ as in Theorem~\ref{theo:bd}, in particular satisfying $\diam(M) \le \diam(N) \le 6^n \diam(M)$.
Then
\[
K_N, K_M \le 6^{2n} \, \varepsilon \cdot (\diam N)^{-2} 
\]
and
\[
\Ric_N, \Ric_M \le - (n-1) \cdot 6^{2n} \lambda \varepsilon \cdot (\diam N)^{-2}.
\]
From Lemma~\ref{lem:LCD}, the generalized $n$-torus~$N$ and its universal cover~$\tilde{N}$ are of $\sqrtR$ class~$(\rho,\kappa)$, where 
\[
\rho = 6^{2n} \, \varepsilon \cdot \diam(N)^{-2}
\]
and 
\[
\kappa = (n-\lambda) \frac{6^{2n}}{n-1}\varepsilon \cdot (\diam N)^{-2}.
\]

We can assume that $\lambda>n$ (otherwise, we are done).
In this case, we have $\kappa <0$.
To get the best from our estimates, we fix
\[
R=\frac{\pi}{2\sqrt{\rho}} = \frac{\pi}{2\cdot 6^n \sqrt{\varepsilon}} \cdot \diam N
\]
Let $r=\frac{R}{5}$ and note that $\frac{R}{4}-r = \frac{R}{20}$.
Fix
\[
\bar\varepsilon_n  = \left( \frac{\pi}{10 \cdot 6^n \cdot a_n} \right)^2 
\]
where $a_n$ is given by Proposition~\ref{prop:rupper}.
Assume that $\varepsilon \in (0,\bar\varepsilon_n)$.
From this choice, we get $R \ge r > a_n \, \diam(N)$ and $\diam(N) \le \frac{R}{5 a_n} = 8^{-n^2}r \le 10^{-8} r$.
Thus,
\[
\frac{R+\diam(N)}{r-\diam(N)} \le 6.
\]
and from Proposition~\ref{prop:rupper} we obtain
\[\frac{\vol B(\tilde{x}_0,R)}{\vol B(\tilde{x}_0,r)}\le D_n\cdot 6^n
\]
where $\tilde{x}_0$ is any point in $\tilde N$.

Proposition~\ref{prop:rcomp} (applied to~$\tilde{N}$) and Proposition~\ref{prop:rlower} then show that
\[
\frac{5^n}{2^{2n-1}} \cdot e^{\frac{(n-1)}{20} \sqrt{|\kappa|} \, R} \le D_n \cdot 6^n.
\]
Thus, 
\[
|\kappa| \, R^2 \le E_n
\]
for some explicit positive constant~$E_n$, which can be taken equal to $7000 \, n^4$.

As, under the assumption $\lambda>n$, we have
\[
|\kappa| \, R^2 =\frac{ (\lambda - n)\pi^2}{4(n-1)},
\]
a rough estimate yields the explicit upper bound $\lambda \le 3000 \, n^5$.

Last, we can strengthen $\bar\varepsilon_n$ into a simpler expression $\varepsilon_n = 2^{-6n^2-7n}$.
\end{proof}

\end{document}